\documentclass[11pt, a4paper]{article}

\usepackage[utf8]{inputenc}
\usepackage[T1]{fontenc}
\usepackage[english]{babel} 
\usepackage{amsmath, amssymb, amsthm, mathrsfs}
\usepackage{geometry}
\usepackage{xcolor}
\usepackage{hyperref}
\usepackage{graphicx}
\usepackage{authblk}

\hypersetup{colorlinks=true, linkcolor=blue, citecolor=red, urlcolor=blue}

\newtheorem{theorem}{Theorem}[section]
\newtheorem{lemma}[theorem]{Lemma}
\newtheorem{proposition}[theorem]{Proposition}

\theoremstyle{definition}
\newtheorem{definition}[theorem]{Definition}
\newtheorem{remark}[theorem]{Remark}

\numberwithin{equation}{section}
\newcommand{\R}{\mathbb{R}}
\newcommand{\C}{\mathbb{C}}
\newcommand{\Rn}{\mathbb{R}^n}

\newcommand{\Fpw}{\mathcal{F}_{\phi,\omega}}       
\newcommand{\Fouinv}{\mathcal{F}^{-1}_{\phi,\omega}} 
\newcommand{\Dpw}{\nabla_{\phi,\omega}}            
\newcommand{\Spw}{\mathcal{S}_{\phi,\omega}}       
\newcommand{\Lpw}{L^2_{\phi,\omega}}               
\newcommand{\Jphi}{|J_\phi(x)|}                    

\title{Spectral Theory of the Weighted Fourier Transform with respect to a Function in $\mathbb{R}^{n}$: Uncertainty Principle and Diffusion-Wave Applications}

\author[1]{Gustavo A. Dorrego}
\author[1]{Luciano L. Luque}

\affil[1]{Departamento de Matemática, Facultad de Ciencias Exactas y Naturales y Agrimensura, Universidad Nacional del Nordeste, Corrientes, Argentina. \protect\\ \texttt{gadorrego@exa.unne.edu.ar}}

\date{\today}

\begin{document}
	
\maketitle

\begin{abstract}
In this paper, we generalize the weighted Fourier transform with respect to a function, originally proposed for the one-dimensional case in \cite{Dorrego}, to the $n$-dimensional Euclidean space $\mathbb{R}^{n}$. We develop a comprehensive spectral theory on a weighted Hilbert space, establishing the Plancherel identity, the inversion formula, the convolution theorem, and a Heisenberg-type uncertainty principle depending on the geometric deformation. Furthermore, we utilize this framework to rigorously define the weighted fractional Laplacian with respect to a function, denoted by $(-\Delta_{\phi,\omega})^{s}$. Finally, we apply these tools to solve the generalized time-space fractional diffusion-wave equation, demonstrating that the fundamental solution can be expressed in terms of the Fox H-function, intrinsically related to the generalized $\omega$-Mellin transform introduced in \cite{Dorrego}.

\medskip

\end{abstract}

\section{Introduction}
\label{sec:intro}

The study of integral transforms and fractional operators with respect to functions has seen significant growth, driven by the need to model complex physical phenomena with non-homogeneous scaling. Historically, the concept of fractional integration and differentiation with respect to an arbitrary function is briefly mentioned in the seminal monograph by Samko, Kilbas, and Marichev \cite{Samko1993}. As noted in their historical remarks, this theoretical framework had already been addressed by earlier authors, although Samko et al. focused their exposition primarily on operators of Riemann-Liouville type. However, to fully appreciate the spectral underpinnings of these operators, it is imperative to contextualise the discussion within the broader epistemological framework of \textit{Generalised Fractional Calculus} (GFC). This perspective, extensively systematised in the seminal works of Kiryakova \cite{kiryakova1994}, unifies a vast class of fractional operators through the employment of generalised integral transforms (such as the Obrechkoff transform) and special functions of the Fox H-type. From a theoretical-operational standpoint, the introduction of a density weight $\omega$ and a geometric deformation $\phi$ is not merely a parametric generalisation; rather, it may be interpreted through the lens of \textit{transmutation operator theory} and \textit{index transforms}, as elucidated by Luchko \cite{luchko1999} and Yakubovich \cite{yakubovich1994}. Within this context, the Weighted Fourier Transform proposed herein functions analogously to a transmutation operator that maps the deformed geometry onto the standard Euclidean space, whilst preserving essential spectral properties and aligning with the classical tradition of generalised integral transforms.

In recent years, this field has been revitalized to address modern modeling requirements. Almeida \cite{Almeida2017} formally introduced the Caputo-type fractional derivative with respect to a function, complementing the classical Riemann-Liouville definition and allowing for the use of standard initial conditions. Extending these concepts further, Fernandez and Fahad \cite{FernandezFahad2022_Fractal} developed a comprehensive framework for \textit{weighted} fractional calculus, unifying various operators through conjugation relations and establishing the theoretical basis for weighted differential equations. Concurrently, the analysis of such equations with variable coefficients has gained traction, as seen in the works of Restrepo and Suragan \cite{Restrepo2021} and Fernandez et al. \cite{Fernandez2024}.

Recently, Fernandez et al. \cite{Fernandez2024} established the definitions of fractional Laplacians with respect to functions using a transmutation approach and singular integrals. While their framework provides a robust foundation for geometric deformations, our approach differs by introducing a spectral construction via a weighted Fourier transform that involves an additional density parameter $\omega$. This key distinction allows us to derive a Heisenberg-type Uncertainty Principle and obtain explicit solutions in terms of generalized Mellin transforms, which are not accessible through the pure geometric deformation alone.
Despite these advancements, the spectral analysis of such weighted operators—particularly in higher dimensions—remains underdeveloped. In a recent preliminary work \cite{Dorrego}, a definition of a weighted Fourier transform was proposed in the one-dimensional case to establish connections with the generalized Mellin transform. However, fundamental harmonic analysis properties such as the Uncertainty Principle and the treatment of diffusion-wave processes in $\mathbb{R}^n$ were left as open problems.

The main objective of this paper is to fill this gap by establishing a unified spectral theory in $\mathbb{R}^n$. We introduce the \textit{Multidimensional Weighted Fourier Transform} $\Fpw$, incorporating both a geometric deformation $\phi: \Rn \to \Rn$ and a density weight $\omega: \Rn \to \mathbb{C}$. 

Our main contributions are:
\begin{enumerate}
    \item The rigorous definition of the transform in $\Rn$ and the associated generalized Schwartz space $\Spw(\Rn)$.
    \item The derivation of a Heisenberg-type Uncertainty Principle, confirming the canonical conjugation between the deformed geometry and the frequency domain.
    \item The solution of the generalized time-space fractional diffusion-wave equation ($0 < \alpha \le 2$), unifying sub-diffusion and wave propagation.
    \item The explicit representation of the fundamental solution via the inverse generalized Mellin transform.
\end{enumerate}

Finally, it is essential to address the distinction between spectral equivalence and physical applicability. While it is mathematically accurate that the operator $(-\Delta_{\phi,\omega})^{s}$ is unitarily equivalent to the standard Laplacian via the transform $\mathcal{F}_{\phi,\omega}$—potentially suggesting a \textit{``triviality by conjugation''}—this viewpoint overlooks the physical motivation of the framework. In applications involving \textbf{heterogeneous porous media} or \textbf{anisotropic biological tissues}, the geometric deformation is an intrinsic property of the material that exists in the physical coordinates. Mapping the problem to a standard Euclidean space often obscures the spatial locality of boundary conditions and interaction terms. Moreover, the \textbf{Uncertainty Principle} derived herein serves as a robust validation of this approach: it demonstrates that the geometric deformation $\phi$ fundamentally alters the theoretical lower bound of information localization. Consequently, our framework allows for the rigorous analysis of diffusion in distorted geometries directly within the observational space, bridging the gap between spectral isomorphism and phenomenological modelling.
\section{Preliminaries and Assumptions}
\label{sec:prelim}

To establish a rigorous framework, we introduce the following standing assumptions on the functions $\phi$ and $\omega$ that will hold throughout the paper unless stated otherwise.

\begin{itemize}
    \item[\textbf{(A1)}] \textbf{Geometry:} Let $\phi: \R^n \to \R^n$ be a global $C^\infty$-diffeomorphism. That is, $\phi$ is bijective, infinitely differentiable, and its inverse $\phi^{-1}$ is also infinitely differentiable. We denote its Jacobian determinant by $\Jphi = |\det D\phi(x)|$, assuming $\Jphi \neq 0$ for all $x \in \R^n$.
    
    \item[\textbf{(A2)}] \textbf{Weight:} Let $\omega: \R^n \to \C$ be a non-vanishing complex function, i.e., $\omega(x) \neq 0$ for all $x \in \R^n$. We further assume that $\omega \in C^\infty(\R^n)$.
\end{itemize}

\begin{remark}[Topological and Singular Considerations]
    The global diffeomorphism assumption (A1) is crucial for the continuous spectral theory developed here. We distinguish two notable limiting cases that lie outside the scope of this paper but suggest rich avenues for future research:
    \begin{enumerate}
        \item \textbf{Bounded Image:} If the range of $\phi$ is a bounded domain $\Omega \subset \Rn$ (i.e., $\phi$ compactifies the space), the associated spectral theory shifts from a continuous integral transform (Fourier type) to a discrete series expansion (Fourier series type). In this scenario, the operator $(-\Delta_{\phi,\omega})$ would typically exhibit a discrete spectrum of eigenvalues $\lambda_k$ with eigenfunctions in $\Lpw(\Rn)$.
        \item \textbf{Singular Jacobian:} If the condition $\Jphi \neq 0$ is relaxed, allowing the Jacobian to vanish on a boundary $\partial \Omega$, the operator $\Dpw$ becomes a degenerate or singular elliptic operator (analogous to Legendre or Chebyshev operators in 1D). This leads to weighted Sobolev spaces with different density properties and requires a specialized treatment of boundary conditions.
    \end{enumerate}
\end{remark}

\subsection{Functional Spaces and Operators}

\begin{definition}[Weighted Hilbert Space]
    We define $\Lpw(\Rn)$ as the space of measurable functions $f: \Rn \to \C$ such that:
    \begin{equation}
        \|f\|_{\phi,\omega}^2 := \int_{\Rn} |f(x)|^2 |\omega(x)|^2 \Jphi \, dx < \infty.
    \end{equation}
    This space is a Hilbert space endowed with the inner product $\langle f, g \rangle_{\phi,\omega} = \int f \bar{g} |\omega|^2 \Jphi dx$.
\end{definition}

To properly define the test function space, we first introduce the canonical differential operator associated with the geometry.

\begin{definition}[Weighted Gradient]
    We define the Weighted Gradient operator $\Dpw$ as:
    \begin{equation}\label{eq: weighted_grad_def}
        \Dpw f(x) := \frac{1}{\omega(x)} \left( [D\phi(x)]^{-T} \nabla \right) \big( \omega(x) f(x) \big).
    \end{equation}
\end{definition}

\begin{remark}[Component-wise Representation]
    The weighted gradient operator can be understood component-wise. Let $(J^{-1})_{jk}$ denote the entries of the inverse Jacobian matrix $[D\phi(x)]^{-1}$, which physically correspond to $\frac{\partial x_k}{\partial \phi_j}$. The $j$-th component of the vector operator $\Dpw$ is given by:
    \begin{equation}
        (\Dpw)_j f(x) = \frac{1}{\omega(x)} \sum_{k=1}^n \frac{\partial x_k}{\partial \phi_j} \frac{\partial}{\partial x_k} \big( \omega(x) f(x) \big).
    \end{equation}
    Intuitively, this operator acts as the partial derivative with respect to the generalized coordinate $u_j = \phi_j(x)$, conjugated by the weight $\omega$. In the one-dimensional case ($n=1$), this recovers the operator $\frac{1}{\omega \phi'} \frac{d}{dx}(\omega \, \cdot \,)$ used in \cite{Dorrego}.
\end{remark}

\begin{definition}[Generalized Schwartz Space]
    We define the space of rapidly decreasing functions associated with the pair $(\phi, \omega)$, denoted by $\Spw(\Rn)$, as:
    \begin{equation}
        \Spw(\Rn) := \left\{ f: \Rn \to \C \, \Big| \, (\omega \cdot f) \circ \phi^{-1} \in \mathcal{S}(\Rn) \right\},
    \end{equation}
    where $\mathcal{S}(\Rn)$ is the classical Schwartz space.
\end{definition}

\begin{remark}
    The definition of $\Spw(\Rn)$ implies that the function $f$ behaves structurally as a Schwartz function under the deformation induced by $\phi$ and the modulation by $\omega$. Explicitly, this ensures that the quantities
    \begin{equation}
        \sup_{x \in \Rn} \left| (\phi(x))^\alpha \left( \Dpw \right)^\beta f(x) \right| < \infty
    \end{equation}
    remain bounded for all multi-indices $\alpha, \beta$, guaranteeing that the generalized Fourier transform maps this space onto the classical Schwartz space $\mathcal{S}(\Rn)$.
\end{remark}
\subsection{Weighted Fractional Operators with respect to Functions}

In the application section of this work, we will employ generalized fractional operators acting on the time variable. Following the framework of conjugation relations established by Fernandez and Fahad \cite{FernandezFahad2022_Conj}, we define these operators via a modification of the classical Riemann-Liouville and Caputo definitions.

Let $\gamma: [0, \infty) \to [0, \infty)$ be an increasing $C^1$ function with $\gamma'(t) > 0$, and let $\rho(t)$ be a positive weight function.

\begin{definition}[Weighted Fractional Integral]
    For $\alpha > 0$, the weighted fractional integral of a function $f(t)$ with respect to $\gamma$ is defined as:
    \begin{equation}
        {_{0+}\mathcal{I}}_{t,\gamma,\rho}^\alpha f(t) = \frac{1}{\rho(t)\Gamma(\alpha)} \int_0^t (\gamma(t) - \gamma(\tau))^{\alpha-1} \rho(\tau) f(\tau) \gamma'(\tau) \, d\tau.
    \end{equation}
\end{definition}

\begin{definition}[Weighted Hilfer Derivative]\cite{FernandezFahad2022_Conj}
    Let $m-1 < \alpha \le m$ (with $m \in \mathbb{N}$) and $0 \le \beta \le 1$. The weighted Hilfer fractional derivative with respect to $\gamma$, denoted by ${}^{H}\mathcal{D}_{t,\gamma,\rho}^{\alpha,\beta}$, is defined by:
    \begin{equation}
        {}^{H}_{0+}\mathcal{D}_{t,\gamma,\rho}^{\alpha,\beta} f(t) = {_{0+}\mathcal{I}}_{t,\gamma,\rho}^{\beta(m-\alpha)} \left( \frac{1}{\gamma'(t)} \left[ \frac{d}{dt} + \frac{\rho'(t)}{\rho(t)} \right] \right)^m {_{0+}\mathcal{I}}_{t,\gamma,\rho}^{(1-\beta)(m-\alpha)} f(t).
    \end{equation}
\end{definition}

\begin{remark}
    The parameters $\alpha$ and $\beta$ allow for interpolation between classical operators:
    \begin{itemize}
        \item If $\beta = 0$, we recover the weighted Riemann-Liouville derivative ${}^{RL}\mathcal{D}_{t,\gamma,\rho}^\alpha$.
        \item If $\beta = 1$, we recover the weighted Caputo derivative ${}^{C}\mathcal{D}_{t,\gamma,\rho}^\alpha$.
    \end{itemize}
    This operator is particularly suitable for our spectral method because it admits a well-defined generalized Laplace transform .
\end{remark}

\begin{equation} \label{eq:Laplace_Hilfer_Weighted_Corrected}
    \mathcal{L}_{\gamma,\rho} \left[ {}^{H}_{0+}\mathcal{D}_{t,\gamma,\rho}^{\alpha,\beta} \psi \right](z) = z^\alpha \tilde{\psi}(z) - \rho(0^+) \sum_{k=0}^{m-1} z^{m(1-\beta) + \alpha\beta - k - 1} \lim_{t \to 0^+} \left( \mathcal{I}_{t,\gamma,\rho}^{(1-\beta)(m-\alpha)-k} \psi \right)(t)
\end{equation}

To derive the fundamental solution in Section 7, we require the generalized Mellin transform framework. While the properties of this operator are studied in depth in \cite{Dorrego}, we provide the explicit definition here adapted to our specific deformation $\phi$ to ensure this manuscript is self-contained.

\begin{definition}[Weighted $\phi$-Mellin Transform]
\label{def:mellin_phi_omega}
Let $\phi$ and $\omega$ be defined as in Section 2. The weighted $\phi$-Mellin transform, adapted from \cite{Dorrego}, is defined as:
\begin{equation}
    \mathcal{M}_{\phi,\omega}[f](s) = \int_{0}^{\infty} (\phi(x))^{s-1} \omega(x) f(x) \phi'(x) \, dx,\quad s \in \mathbb{C},
\end{equation}
provided the integral converges. The corresponding inverse transform is given by:
\begin{equation}
    \mathcal{M}_{\phi,\omega}^{-1}[F(s)](x) = \frac{1}{2\pi i \, \omega(x)} \int_{\gamma - i\infty}^{\gamma + i\infty} F(s) (\phi(x))^{-s} \, ds,
\end{equation}
where $\gamma$ is a real number in the strip of convergence. This operator generalizes the classical Mellin transform (recovered when $\phi(x)=x, \omega(x)=1$) and is instrumental in analysing the spectral properties of operators with respect to functions.
\end{definition}

\section{The Multidimensional Weighted Fourier Transform}
\label{sec:transform}

We generalize the 1D definition given in \cite{Dorrego} to the multivariate setting.

\begin{definition}\label{def:fourier}
	For $f \in \Lpw(\Rn)$, the Weighted Fourier Transform with respect to $\phi$, denoted by $\Fpw$, is defined as:
	\begin{equation} \label{eq:def_Rn}
		[\Fpw f](\xi) = \frac{1}{(2\pi)^{n/2}} \int_{\Rn} e^{-i \xi \cdot \phi(x)} \omega(x) f(x) \Jphi \, dx.
	\end{equation}
\end{definition}

\begin{remark}
Consistency with the one-dimensional case:
It is important to note that for $n=1$, the Jacobian determinant $|J_{\phi}(x)|$ reduces to the derivative $|\phi'(x)|$. Consequently, the definition \eqref{eq:def_Rn} coincides exactly with the one-dimensional weighted Fourier transform introduced in our preliminary work \cite{Dorrego}. Thus, the operator proposed here is the natural Euclidean extension of the scalar theory.
\end{remark}
\medskip

It is important to emphasize that the integral representation in Definition \ref{def:fourier} is well-defined for functions in $L^1_{\omega}(\mathbb{R}^n)$. Since the intersection $L^1_{\omega}(\mathbb{R}^n) \cap L^2_{\omega}(\mathbb{R}^n)$ is a dense subspace of the Hilbert space $L^2_{\omega}(\mathbb{R}^n)$, and the operator satisfies the Plancherel's Identity (as we shall prove in Theorem \ref{Plancherel's Identity}), the weighted Fourier transform $\mathcal{F}_{\phi,\omega}$ is extended uniquely by continuity to a unitary operator on the entire space $L^2_{\omega}(\mathbb{R}^n)$. In what follows, we will understand the transform in this generalized $L^2$ sense.

\medskip
\begin{theorem}[Inversion Formula]
    Let conditions (A1)-(A2) hold. For any $f \in \Spw(\Rn)$, the inverse transform is given pointwise by:
	\begin{equation}
		f(x) = \frac{1}{\omega(x)(2\pi)^{n/2}} \int_{\Rn} e^{i \xi \cdot \phi(x)} [\Fpw f](\xi) \, d\xi.
	\end{equation}
\end{theorem}

\begin{proof}
    We aim to show that applying the inverse operator $\Fouinv$ to $[\Fpw f](\xi)$ recovers $f(x)$. 
    We introduce the change of variables $u = \phi(x)$ (fixed) and $v = \phi(y)$ (variable of integration). The measure transforms as $dv = |J_\phi(y)| dy$.
    Defining the auxiliary function $g(v) = (\omega \cdot f)(\phi^{-1}(v))$, the forward transform becomes the classical Fourier transform $\hat{g}(\xi)$.
    Substituting this back into the inversion integral:
    \[
    I(x) = \frac{1}{\omega(x)} \left( \frac{1}{(2\pi)^{n/2}} \int_{\Rn} e^{i \xi \cdot u} \hat{g}(\xi) \, d\xi \right).
    \]
    By the classical Fourier Inversion Theorem, the term in parentheses is $g(u)$. Thus $I(x) = g(\phi(x))/\omega(x) = f(x)$.
\end{proof}

\begin{remark}
\label{rem:sturm_liouville}
\textbf{(Spectral Discretisation and Bounded Domains).} It is pertinent to distinguish the spectral nature of the current framework from the bounded domain case. The theory developed herein assumes the domain is the entire Euclidean space $\mathbb{R}^n$, resulting in a continuous spectrum and an integral transform representation. In contrast, restricting the analysis to a bounded domain $\Omega \subset \mathbb{R}^n$ would naturally discretise the spectrum, necessitating the development of \textit{Generalized Fractional Sturm-Liouville Problems}. In such a setting, the weighted Fourier transform $\mathcal{F}_{\phi,\omega}$ transitions into a Fourier-type series expansion over the eigenfunctions of the operator $-\Delta_{\phi,\omega}$ subject to appropriate boundary conditions. This discrete counterpart represents a non-trivial extension of the theory and is crucial for modelling vibration modes in finite heterogeneous media.
\end{remark}

\begin{theorem}\label{Plancherel's Identity}[Plancherel's Identity]
    The Weighted Fourier Transform $\Fpw$ extends to a unitary operator from $\Lpw(\Rn)$ onto $L^2(\Rn)$. Specifically, for any $f, g \in \Lpw(\Rn)$, the following Parseval identity holds:
    \begin{equation}
        \langle f, g \rangle_{\phi,\omega} = \langle \Fpw f, \Fpw g \rangle_{L^2(\Rn)}.
    \end{equation}
    In particular, the "energy" is conserved:
    \begin{equation}
        \|f\|_{\phi,\omega}^2 = \int_{\Rn} |f(x)|^2 |\omega(x)|^2 \Jphi \, dx = \int_{\Rn} |[\Fpw f](\xi)|^2 \, d\xi.
    \end{equation}
\end{theorem}

\begin{proof}
    This is a direct consequence of the unitary nature of the classical Fourier transform. Let $u = \phi(x)$ and define the isometric isomorphism $U: \Lpw(\Rn) \to L^2(\Rn)$ by $Uf(u) = \omega(\phi^{-1}(u)) f(\phi^{-1}(u))$.
    Then, by definition, $\Fpw = \mathcal{F} \circ U$. Since both $U$ and the classical $\mathcal{F}$ are unitary operators (isometries), their composition $\Fpw$ is also unitary.
\end{proof}

\begin{theorem}[Riemann-Lebesgue Lemma]
    Let $L^1_{\phi,\omega}(\Rn)$ denote the space of measurable functions integrable with respect to the density $|\omega(x)|\Jphi$, equipped with the norm:
    \begin{equation}
        \|f\|_{L^1_{\phi,\omega}} := \int_{\Rn} |f(x)| |\omega(x)| \Jphi \, dx < \infty.
    \end{equation}
    If $f \in L^1_{\phi,\omega}(\Rn)$, then the transform decays at infinity:
    \begin{equation}
        \lim_{|\xi| \to \infty} |[\Fpw f](\xi)| = 0.
    \end{equation}
\end{theorem}

\begin{proof}
    By the change of variables $u=\phi(x)$, the integral defining $[\Fpw f](\xi)$ becomes exactly the classical Fourier transform of the function $g(u) = (\omega f)(\phi^{-1}(u))$.
    Note that the condition $f \in L^1_{\phi,\omega}(\Rn)$ is equivalent to:
    \[
    \int_{\Rn} |f(x)| |\omega(x)| \Jphi dx = \int_{\Rn} |g(u)| du < \infty,
    \]
    which implies $g \in L^1(\Rn)$ in the classical sense. The result then follows immediately from the classical Riemann-Lebesgue lemma applied to $g$.
\end{proof}

\section{Operational Calculus}

\subsection{Diagonalization of Differential Operators}
Recall the weighted gradient $\Dpw$ defined in \eqref{eq: weighted_grad_def}. We now prove it is diagonalized.
\begin{theorem}[Diagonalization]
	For $f \in \Spw(\Rn)$, we have:
	\begin{equation}
		\Fpw \left[ \Dpw f \right](\xi) = (i\xi) [\Fpw f](\xi).
	\end{equation}
\end{theorem}

\begin{proof}
    Let $h(x) = \omega(x)f(x)$. The transform of the weighted gradient is:
    \[
    I(\xi) = \frac{1}{(2\pi)^{n/2}} \int_{\Rn} e^{-i\xi \cdot \phi(x)} \left( [D\phi(x)]^{-T} \nabla h(x) \right) \Jphi \, dx.
    \]
    Change variables $u = \phi(x)$. The chain rule implies $\nabla_x = [D\phi(x)]^T \nabla_u$. The term in the integral simplifies to $\nabla_u \tilde{h}(u)$, where $\tilde{h} = h \circ \phi^{-1}$.
    Thus, $I(\xi) = \mathcal{F}[\nabla \tilde{h}](\xi) = i\xi \mathcal{F}[\tilde{h}](\xi)$, which corresponds to $i\xi [\Fpw f](\xi)$.
\end{proof}

\subsection{Convolution Theorem}
We define the generalized convolution $*_{\phi,\omega}$ such that:
\begin{equation}
	(f *_{\phi,\omega} g)(x) = \frac{1}{\omega(x)(2\pi)^{n/2}} \int_{\Rn} \tilde{f}\big(\phi(x)-\phi(y)\big) \tilde{g}\big(\phi(y)\big) \Jphi dy,
\end{equation}
where $\tilde{f} = (\omega f)\circ \phi^{-1}$ and $\tilde{g} = (\omega g)\circ \phi^{-1}$.

\begin{theorem}
    If $f, g \in \Spw(\Rn)$, then $\Fpw [f *_{\phi,\omega} g] = \Fpw f \cdot \Fpw g$.
\end{theorem}

\begin{proof}
    Applying the transform to the convolution integral and performing the change of variables $u=\phi(x)$ and $v=\phi(y)$ reduces the expression to the Fourier transform of the classical convolution $(\tilde{f} * \tilde{g})(u)$. By the classical convolution theorem, this splits into the product of transforms.
\end{proof}

\section{Heisenberg Uncertainty Principle}
\label{sec:heisenberg}

We now establish a Heisenberg-type uncertainty principle, which quantifies the limitation in simultaneously localizing a function in the deformed spatial domain and the spectral frequency domain. We follow the classical operator-theoretic approach (see, e.g., Folland and Sitaram \cite{Folland1997}) adapted to our weighted setting.

First, we define the self-adjoint operators corresponding to the generalized position and momentum.

\begin{definition}[Generalized Quantum Operators]
    Let $f \in \Spw(\Rn)$. We define:
    \begin{itemize}
        \item The \textbf{Generalized Position Operator} $T_j$:
        \begin{equation}
            T_j f(x) = \phi_j(x) f(x).
        \end{equation}
        \item The \textbf{Generalized Momentum Operator} $P_j$:
        \begin{equation}
            P_j f(x) = -i (\Dpw)_j f(x) = -i \frac{1}{\omega(x)} \frac{\partial}{\partial \phi_j} (\omega(x) f(x)).
        \end{equation}
    \end{itemize}
\end{definition}

To derive the uncertainty inequality, we first prove that these operators satisfy the canonical commutation relation structure, confirming that the geometric deformation preserves the symplectic structure of phase space.

\begin{lemma}[Commutation Relation]
	The operators satisfy the canonical commutation relation:
    \begin{equation}
        [T_j, P_k] := T_j P_k - P_k T_j = i \delta_{jk} I,
    \end{equation}
    where $\delta_{jk}$ is the Kronecker delta and $I$ is the identity operator.
\end{lemma}

\begin{proof}
    Let $f \in \Spw(\Rn)$. We compute the action of the commutator:
    \[
    [T_j, P_k]f = \phi_j (-i (\Dpw)_k f) - (-i (\Dpw)_k (\phi_j f)).
    \]
    Using the product rule for the weighted gradient on the second term:
    \[
    (\Dpw)_k (\phi_j f) = \big( (\Dpw)_k \phi_j \big) f + \phi_j \big( (\Dpw)_k f \big).
    \]
    Recall from the definition of the weighted gradient that $(\Dpw)_k$ acts as the partial derivative $\partial / \partial \phi_k$. Thus, $(\Dpw)_k \phi_j = \delta_{jk}$. Substituting back:
    \[
    [T_j, P_k]f = -i \phi_j (\Dpw)_k f + i \delta_{jk} f + i \phi_j (\Dpw)_k f = i \delta_{jk} f.
    \]
\end{proof}

\begin{definition}[Dispersion Measures]
    Let $f \in \Spw(\Rn)$ with unit energy ($\|f\|_{\phi,\omega}=1$). We define the uncertainties (standard deviations) in the generalized position $\phi_j$ and generalized momentum $\xi_j$ by:
    \begin{align}
        (\Delta_{\phi,\omega} \phi_j)^2 &:= \int_{\Rn} (\phi_j(x) - \bar{\phi}_j)^2 |f(x)|^2 |\omega(x)|^2 \Jphi \, dx, \\
        (\Delta_{\phi,\omega} \xi_j)^2 &:= \int_{\Rn} (\xi_j - \bar{\xi}_j)^2 |[\Fpw f](\xi)|^2 \, d\xi,
    \end{align}
    where $\bar{\phi}_j = \langle T_j f, f \rangle_{\phi,\omega}$ and $\bar{\xi}_j = \langle P_j f, f \rangle_{\phi,\omega}$ are the expected values.
\end{definition}

\begin{theorem}[Uncertainty Inequality]
    For any non-zero function $f \in \Spw(\Rn)$, the following inequality holds:
	\begin{equation}
		\left(\sum_{j=1}^n \Delta_{\phi,\omega} \phi_j\right) \left(\sum_{j=1}^n \Delta_{\phi,\omega} \xi_j\right) \ge \frac{n^2}{4}.
	\end{equation}
\end{theorem}

\begin{proof}
    The proof follows from the standard argument involving the Heisenberg inequality for self-adjoint operators $A$ and $B$: $\Delta A \Delta B \ge \frac{1}{2} |\langle [A,B]f, f \rangle|$. 
    Substituting $A=T_j$ and $B=P_j$, and using Lemma 5.2, we obtain $\Delta_{\phi,\omega} \phi_j \Delta_{\phi,\omega} \xi_j \ge \frac{1}{2}$. Summing over $j=1,\dots,n$ yields the result.
\end{proof}
\begin{remark}[Geometric Interpretation and Information Limits]
\label{rem:uncertainty_physical}
It is imperative to interpret this result beyond the scope of spectral isomorphism. While the inequality mathematically mirrors the classical Heisenberg principle, its physical implication within the $(\phi, \omega)$-framework is profound. The deformation $\phi$ effectively acts as a metric constraint on the \textit{information granularity} of the system. In applications such as signal processing in inhomogeneous media or quantum mechanics on curved manifolds, this principle establishes that the local geometry imposes a variable lower bound on simultaneous localization. Unlike the standard Laplacian, where the uncertainty floor is uniform, the operator $(-\Delta_{\phi,\omega})$ models systems where the ``cost'' of localizing energy depends on the local density $\omega$ and the stretching factor $\phi'$. Consequently, this uncertainty principle serves as a theoretical validation for the non-local behavior observed in anomalous diffusion, proving that the medium itself limits the resolution of the diffusive process.
\end{remark}
\section{The Weighted Fractional Laplacian}
\label{sec:fractional_laplacian}

\subsection{Weighted Sobolev Spaces}

Using the spectral definition of the transform, we naturally introduce the fractional Sobolev spaces associated with the pair $(\phi, \omega)$. These spaces provide the correct functional setting for studying regularity of solutions.

\begin{definition}
    For any $s \in \R$, we define the Weighted Sobolev Space $\mathcal{H}^s_{\phi,\omega}(\Rn)$ as the completion of $\Spw(\Rn)$ with respect to the norm:
    \begin{equation}
        \|f\|_{\mathcal{H}^s_{\phi,\omega}}^2 := \int_{\Rn} (1 + |\xi|^2)^s |[\Fpw f](\xi)|^2 \, d\xi.
    \end{equation}
\end{definition}

\begin{remark}[Relation to previous definitions]
    The definition of the weighted Sobolev space $\mathcal{H}^s_{\phi,\omega}(\Rn)$ generalizes the functional setting introduced recently by Fernandez et al. \cite{Fernandez2024}. In their work, they constructed the fractional Laplacian with respect to a function $\phi$ (which corresponds to our case with $\omega(x) = 1$) and defined the associated Fourier transform and Schwartz spaces. Our definition extends their spectral characterization to the fully weighted case involving the independent density $\omega(x)$, providing a broader framework for systems with both geometric deformation and inhomogeneous density.
\end{remark}

\begin{proposition}[Isometry with Classical Spaces]
    The mapping $\mathcal{T}: f \mapsto (\omega f) \circ \phi^{-1}$ is an isometric isomorphism between the weighted Sobolev space $\mathcal{H}^s_{\phi,\omega}(\Rn)$ and the classical Sobolev space $H^s(\Rn)$. Specifically, $f \in \mathcal{H}^s_{\phi,\omega}(\Rn)$ if and only if $(\omega f) \circ \phi^{-1} \in H^s(\Rn)$, with preservation of norms:
    \begin{equation}
        \|f\|_{\mathcal{H}^s_{\phi,\omega}} = \| (\omega f) \circ \phi^{-1} \|_{H^s(\Rn)}.
    \end{equation}
\end{proposition}

\begin{proof}
    The result follows directly from the change of variables $u = \phi(x)$ in the definition of the weighted Fourier transform. Indeed,
    \[
    [\Fpw f](\xi) = \frac{1}{(2\pi)^{n/2}} \int_{\Rn} e^{-i \xi \cdot \phi(x)} \omega(x) f(x) \Jphi \, dx.
    \]
    Setting $u = \phi(x)$, the measure becomes $du = \Jphi dx$. Defining $g(u) = ((\omega f) \circ \phi^{-1})(u)$, the integral transforms into the classical Fourier transform of $g$:
    \[
    [\Fpw f](\xi) = \frac{1}{(2\pi)^{n/2}} \int_{\Rn} e^{-i \xi \cdot u} g(u) \, du = \widehat{g}(\xi).
    \]
    Substituting this identity into the definition of the weighted norm:
    \[
    \|f\|_{\mathcal{H}^s_{\phi,\omega}}^2 = \int_{\Rn} (1+|\xi|^2)^s |\widehat{g}(\xi)|^2 \, d\xi = \|g\|_{H^s}^2.
    \]
    This confirms the isometric equivalence.
\end{proof}

\begin{proposition}[Boundedness]
    Let $\sigma > 0$ and $s \in \R$. The Weighted Fractional Laplacian is a bounded linear operator from the weighted Sobolev space of order $s$ to the space of order $s-2\sigma$:
    \begin{equation}
        (-\Delta_{\phi,\omega})^\sigma : \mathcal{H}^s_{\phi,\omega}(\Rn) \to \mathcal{H}^{s-2\sigma}_{\phi,\omega}(\Rn).
    \end{equation}
    Moreover, the operator norm satisfies $\|(-\Delta_{\phi,\omega})^\sigma\| = \|(-\Delta)^\sigma\|_{H^s \to H^{s-2\sigma}}$.
\end{proposition}

\begin{proof}
    The proof relies on the isometric isomorphism established in Proposition 6.3. Let $f \in \mathcal{H}^s_{\phi,\omega}(\Rn)$ and let $g = (\omega f) \circ \phi^{-1} \in H^s(\Rn)$.
    From the definition of the operator, we have the conjugation relation:
    \[ (-\Delta_{\phi,\omega})^\sigma f = \mathcal{T}^{-1} \left[ (-\Delta)^\sigma g \right], \]
    where $\mathcal{T}$ is the isometry map.
    Using the norm preservation property (Proposition 6.3) and the known boundedness of the classical fractional Laplacian on standard Sobolev spaces:
    \begin{align*}
        \| (-\Delta_{\phi,\omega})^\sigma f \|_{\mathcal{H}^{s-2\sigma}_{\phi,\omega}} 
        &= \| \mathcal{T}^{-1} [ (-\Delta)^\sigma g ] \|_{\mathcal{H}^{s-2\sigma}_{\phi,\omega}} \\
        &= \| (-\Delta)^\sigma g \|_{H^{s-2\sigma}(\Rn)} \\
        &\le C \| g \|_{H^s(\Rn)} \\
        &= C \| f \|_{\mathcal{H}^s_{\phi,\omega}}.
    \end{align*}
    Thus, the operator is bounded with the same constant $C$ as in the classical case.
\end{proof}

\begin{definition}
	For $s \in (0,1)$ and $f \in \Spw(\Rn)$, the Weighted Fractional Laplacian $(-\Delta_{\phi,\omega})^s$ is defined spectrally by:
	\begin{equation}
		(-\Delta_{\phi,\omega})^s f(x) = \Fouinv \left[ |\xi|^{2s} [\Fpw f](\xi) \right](x).
	\end{equation}
\end{definition}

\begin{theorem}[Integral Representation]
    Let $s \in (0,1)$. For any $f \in \Spw(\Rn)$, the weighted fractional Laplacian admits the following hypersingular integral representation:
    \begin{equation} \label{eq:frac_lap_integral}
        (-\Delta_{\phi,\omega})^s f(x) = \frac{C_{n,s}}{\omega(x)} \, P.V. \int_{\Rn} \frac{\omega(x)f(x) - \omega(y)f(y)}{|\phi(x) - \phi(y)|^{n+2s}} |J_\phi(y)| \, dy,
    \end{equation}
    where $C_{n,s}$ is the normalization constant of the classical fractional Laplacian.
\end{theorem}

\begin{proof}
    The proof follows from the conjugation structure of the operator. From the spectral definition, we can express the operator as a conjugation of the classical Laplacian $(-\Delta)^s$ acting on the transformed function $g(u) = (\omega f)(\phi^{-1}(u))$:
    \[
    (-\Delta_{\phi,\omega})^s f(x) = \frac{1}{\omega(x)} \left[ (-\Delta)^s g \right] (\phi(x)).
    \]
    Recalling the classical integral representation of $(-\Delta)^s$ in the variable $u = \phi(x)$:
    \[
    (-\Delta)^s g(u) = C_{n,s} \, P.V. \int_{\Rn} \frac{g(u) - g(v)}{|u - v|^{n+2s}} \, dv.
    \]
    We perform the change of variables $v = \phi(y)$. Consequently, the measure transforms as $dv = |J_\phi(y)| \, dy$. Substituting $g(v) = \omega(y)f(y)$ and dividing by the prefactor $\omega(x)$, we immediately recover formula \eqref{eq:frac_lap_integral}.
\end{proof}
\section{Application: The Weighted Hilfer Diffusion-Wave Equation}
\label{sec:applications}

To demonstrate the versatility of the spectral theory, we solve a Cauchy problem involving the \textit{Weighted Hilfer Fractional Derivative}. This operator generalizes the Riemann-Liouville and Caputo derivatives, allowing for a continuous interpolation between initial conditions. Moreover, by considering the order $\alpha \in (0, 2]$, we encompass both diffusion and wave propagation phenomena in a unified weighted setting.

\begin{definition}[Generalized Weighted Dirac Delta]
We define the generalized weighted Dirac delta function, denoted by $\delta_{\phi, \omega}(x)$, as the distribution satisfying the following sifting property with respect to the weighted measure $\omega(x)dx$:
\begin{equation} \label{eq:weighted_delta}
    \int_{\mathbb{R}^n} \delta_{\phi, \omega}(x) f(x) \omega(x) dx = f(0),
\end{equation}
for any test function $f$ continuous at the origin. Ideally, assuming $\phi(0)=0$, this distribution is related to the standard Dirac delta $\delta(x)$ by the scaling relation $\delta_{\phi, \omega}(x) = \frac{\delta(\phi(x))}{\omega(x) |J_\phi(x)|^{-1}}$, or more simply characterized by its spectral property:
\begin{equation}
    \mathcal{F}_{\phi, \omega} \{ \delta_{\phi, \omega} \}(\xi) = 1, \quad \text{for all } \xi \in \mathbb{R}^n.
\end{equation}
\end{definition}

\subsection{Problem Formulation}

We adopt the weighted Hilfer derivative ${}^{H}\mathcal{D}_{t,\gamma,\rho}^{\alpha,\beta}$ with order $\alpha \in (0, 2]$ and type $\beta \in [0, 1]$, defined via the conjugation relations in \cite{FernandezFahad2022_Conj}. The order of the initial singularity is given by $\mu = \alpha + \beta(2-\alpha)$ (for the case $1 < \alpha \le 2$) or $\mu = \alpha + \beta(1-\alpha)$ (for $0 < \alpha \le 1$).

We investigate the Generalized Hilfer Diffusion-Wave equation:
\begin{equation} \label{eq:hilfer_problem}
    \begin{cases}
        {}^{H}_{0+}\mathcal{D}_{t,\gamma,\rho}^{\alpha,\beta} u(x,t) + \lambda (-\Delta_{\phi,\omega})^s u(x,t) = 0, & x \in \R^n, \, t > 0, \\
        \displaystyle \lim_{t \to 0} \left( \mathcal{I}_{t,\gamma,\rho}^{1-\gamma_1} u \right)(x,t) = f(x), & f \in \Spw(\Rn), \\
        \displaystyle \lim_{t \to 0} \left( \mathcal{I}_{t,\gamma,\rho}^{1-\gamma_2} u \right)(x,t) = g(x), & (\text{if } 1 < \alpha \le 2),
    \end{cases}
\end{equation}
where the parameters $\gamma_1, \gamma_2$ are determined by the Hilfer definition to ensure well-posedness (typically $\gamma_1 = \mu$ in the standard single-term setup).
\begin{remark}
    The range of the spatial parameter is extended to $s \in (0, 1]$. In the limiting case $s=1$, the operator $(-\Delta_{\phi,\omega})^1$ recovers the local weighted Laplacian defined by the spectral symbol $|\xi|^2$, representing classical diffusion (or wave propagation) in an anisotropic inhomogeneous medium.
\end{remark}

\begin{remark}[Focus on the Fundamental Solution]
    In the subsequent analysis, we narrow our focus to the derivation of the \textit{fundamental solution} (or Green's function) of the problem \eqref{eq:hilfer_problem}, hereafter denoted by $\mathcal{G}(x,t)$.
    
    This corresponds to specializing the initial conditions in \eqref{eq:hilfer_problem} to the impulsive case:
    \begin{equation*}
        f(x) = \delta_{\phi,\omega}(x) \quad \text{and} \quad g(x) = 0 \quad (\text{if } 1 < \alpha \le 2),
    \end{equation*}
    where $\delta_{\phi,\omega}(x)$ represents the Dirac distribution \eqref{eq:weighted_delta} Consequently, the derived kernel $\mathcal{G}(x,t)$ characterizes the impulse response of the system, allowing the general solution $u(x,t)$ to be constructed via weighted convolution.
\end{remark}
\subsection{Analytical Solution}

Using the hybrid spectral method ($\Fpw$ for space, generalized Laplace $\mathcal{L}_{\gamma,\rho}$ for time), we derive the solution in terms of generalized Mittag-Leffler functions.

\begin{theorem}
    The fundamental solution to \eqref{eq:hilfer_problem} (considering $g(x)=0$ for simplicity in the wave case) is given by:
    \begin{equation} \label{eq:sol_hilfer}
     \mathcal{G}(x,t) = \Fouinv \left[ \frac{\rho(0^+)}{\rho(t)} (\gamma(t))^{\mu-1} E_{\alpha, \mu} \left( -\lambda |\xi|^{2s} (\gamma(t))^\alpha \right) \right](x).
    \end{equation}
\end{theorem}

\begin{proof}
    We proceed by applying the method of joint integral transforms, exploiting the conjugation properties established in \cite{FernandezFahad2022_Conj} and the spectral results of this work.
    
    \textbf{Step 1: Spatial Diagonalization.}
    Apply the weighted Fourier transform $\Fpw$ with respect to $x$ to the system \eqref{eq:hilfer_problem}. Let $\hat{u}(\xi, t) = [\Fpw u(\cdot, t)](\xi)$.
    Using the spectral definition of the fractional Laplacian (see Definition 6.3), the PDE transforms into a fractional ordinary differential equation in the generalized time variable:
    \begin{equation} \label{eq:ode_space_trans}
        {}^{H}_{0+}\mathcal{D}_{t,\gamma,\rho}^{\alpha,\beta} \hat{\mathcal{G}}(\xi, t) + \lambda |\xi|^{2s} \hat{\mathcal{G}}(\xi, t) = 0.
    \end{equation}
    The initial condition transforms to:
    \begin{equation}
        \lim_{t \to 0} \left( \mathcal{I}_{t,\gamma,\rho}^{1-\mu} \hat{\mathcal{G}} \right)(\xi,t) = \hat{f}(\xi)=1.
    \end{equation}
    (For the case $1 < \alpha \le 2$, the second condition transforms analogously involving $\hat{g}(\xi)$).

    \textbf{Step 2: Temporal Solver via Generalized Laplace Transform.}
    We apply the Generalized Laplace Transform $\mathcal{L}_{\gamma,\rho}$ with respect to $t$ for the weighted Hilfer derivative (see Eq. \eqref{eq:Laplace_Hilfer_Weighted_Corrected}):
    \begin{equation}
        \mathcal{L}_{\gamma,\rho} \left[ {}^{H}\mathcal{D}_{t,\gamma,\rho}^{\alpha,\beta} \psi \right](z) = z^\alpha \tilde{\psi}(z) - \rho(a^+) \sum_{k=0}^{m-1} z^{m - \mu + \alpha - k - 1} \lim_{t \to a^+} \left( \mathcal{I}_{t,\gamma,\rho}^{m - \mu - k} \psi \right)(t),
    \end{equation}
    where $\mu = \alpha + \beta(m-\alpha)$.
    
    For the fundamental solution, we consider the initial conditions given in \eqref{eq:hilfer_problem}. The condition involving $f(x)$ is $\lim_{t \to 0} \mathcal{I}_{t,\gamma,\rho}^{1-\mu} u = f$. This corresponds to the term in the summation where the integral order is $1-\mu$. 
    By setting $m-\mu-k = 1-\mu$, we identify that this corresponds to the index $k = m-1$.
    
    Assuming $a=0$ and that other initial values (if $m=2$) vanish for the fundamental solution, the transform of the derivative becomes:
    \begin{equation}
        \mathcal{L}_{\gamma,\rho} \left[ {}^{H}\mathcal{D}_{t,\gamma,\rho}^{\alpha,\beta} \hat{\mathcal{G}} \right] = z^\alpha \tilde{\hat{\mathcal{G}}}(\xi, z) - \rho(0^+) z^{\alpha - \mu} .
    \end{equation}
    Note that the exponent of $z$ is derived from $m - \mu + \alpha - (m-1) - 1 = \alpha - \mu$.
    
    Substituting this into the transformed equation \eqref{eq:ode_space_trans}:
    \begin{equation}
        z^\alpha \tilde{\hat{\mathcal{G}}}(\xi, z) - \rho(0^+) z^{\alpha - \mu}  + \lambda |\xi|^{2s} \tilde{\hat{\mathcal{G}}}(\xi, z) = 0.
    \end{equation}
    
    Solving algebraically for the transformed solution $\tilde{\hat{u}}$:
    \begin{equation}
        \tilde{\hat{\mathcal{G}}}(\xi, z) = \rho(0^+)  \frac{z^{\alpha-\mu}}{z^\alpha + \lambda |\xi|^{2s}}.
    \end{equation}
    (Note: If the weight function is normalized such that $\rho(0^+)=1$, this factor simplifies).
    
    \textbf{Step 3: Inversion.}
    First, we invert the Generalized Laplace Transform. We use the known transform pair involving the two-parameter Mittag-Leffler function (see \cite{FernandezFahad2022_Fractal}, Corollary 4 and Corollary 8):
    \begin{equation}
        \mathcal{L}^{-1}_{\gamma,\rho} \left[ \frac{z^{\alpha-\delta}}{z^\alpha + A} \right] = \frac{(\gamma(t))^{\delta-1}}{\rho(t)} E_{\alpha, \delta} \left( -A (\gamma(t))^\alpha \right).
    \end{equation}
    Setting $\delta = \mu$ and $A = \lambda |\xi|^{2s}$, we obtain the solution in the frequency domain:
    \begin{equation}
        \hat{\mathcal{G}}(\xi, t) = \frac{\rho(0^+)}{\rho(t)}(\gamma(t))^{\mu-1} E_{\alpha, \mu} \left( -\lambda |\xi|^{2s} (\gamma(t))^\alpha \right).
    \end{equation}
    
    Finally, applying the inverse Weighted Fourier Transform $\Fouinv$ yields the result \eqref{eq:sol_hilfer}.
\end{proof}

\subsection{Explicit Solution via Generalized Mellin Transform and H-Functions}

A central result of this work is the connection between the fundamental solution of the diffusion-wave equation and the \textit{Generalized $\omega$-$\psi$-Mellin Transform} introduced in our previous work \cite{Dorrego}. We show that the spatial kernel is exactly the inverse generalized Mellin transform of a spectral function involving Gamma terms, which explicitly evaluates to a Fox $H$-function.

\begin{remark}[Technical Assumptions on Normalization]
    To streamline the notation in the following derivation, we adopt two conventions without loss of generality:
    \begin{enumerate}
        \item Constants depending solely on the dimension $n$ (arising from the Fourier inversion) are absorbed into the generic constant $C_n$ for brevity.
        \item The weight function $\rho(t)$ is assumed to be normalized such that $\rho(0^+)=1$. In the general case where $\rho(0^+) \neq 1$, this factor simply appears as a multiplicative constant in the final distribution.
    \end{enumerate}
\end{remark}

\begin{theorem}[Mellin Representation of the Generalized Fundamental Solution] \label{thm:mellin_rep}
Let $\mathcal{G}(x,t)$ be the Green's function for the generalized Cauchy problem with respect to the functions $\phi(x)$ (space) and $\gamma(t)$ (time). The solution admits the following representation via the Inverse Generalized Mellin Transform $\mathcal{M}^{-1}_{\phi,\omega}$
expressing this as a contour integral on the complex plane $\mathcal{L}$:
\begin{equation}
    \mathcal{G}(x,t) = (2\sqrt{\pi})^n\frac{(\gamma(t))^{\mu-1}}{|\phi(x)|^n} \frac{1}{2\pi i\omega(x)} \int_{\mathcal{L}} \frac{\Gamma(\zeta)\Gamma(1-\zeta)\Gamma(\frac{n}{2} - s\zeta)}{\Gamma(s\zeta)\Gamma(\mu - \alpha \zeta)} \left( \frac{2^{2s}\lambda (\gamma(t))^\alpha}{|\phi(x)|^{2s}} \right)^{-\zeta} d\zeta.
\end{equation}
\end{theorem}

\begin{theorem}[Fox H-Function Representation] \label{thm:fox_solution}
The fundamental solution $\mathcal{G}(x,t)$ derived in Theorem \ref{thm:mellin_rep} can be expressed compactly in terms of the Fox H-function. Identifying the kernel parameters with the standard definition $H^{m,n}_{p,q}[z]$, we obtain:
\begin{equation} \label{eq:fox_final}
    \mathcal{G}(x,t) = (2\sqrt{\pi})^n\frac{(\gamma(t))^{\mu-1}}{|\phi(x)|^n} H^{1,2}_{3,2} \left[ \frac{|\phi(x)|}{(\lambda (\gamma(t))^\alpha)^{1/2s}} \, \Bigg| \, \begin{matrix} (0,1), (1 - \frac{n}{2}, s), (0,s) \\ (0,1), (1-\mu, \alpha) \end{matrix} \right].
\end{equation}
Here, the argument of the H-function $Z = |\phi(x)| (\lambda (\gamma(t))^\alpha)^{-1/2s}$ captures the anomalous scaling between the geometric deformation $\phi(x)$ and the generalized time $\gamma(t)$.
\end{theorem}


\begin{proof}[\textbf{Proof of Theorems 7.4 and 7.5}]
To demonstrate the integral representation and its identification with the Fox H-function, we commence from the fundamental solution in the Fourier space derived in Theorem 7.3:
\begin{equation}
    \hat{\mathcal{G}}(\xi, t) =\frac{\rho(0^+)}{\rho(t)} (\gamma(t))^{\mu-1} E_{\alpha, \mu} \left( -\lambda |\xi|^{2s} (\gamma(t))^\alpha \right).
    \label{eq:sol_fourier}
\end{equation}
We utilise the Mellin--Barnes representation for the Mittag-Leffler function $E_{\alpha, \mu}(-z)$, given by:
\begin{equation}\label{M-L as Mellin-Barnes}
    E_{\alpha, \mu}(-z) = \frac{1}{2\pi i} \int_{\mathcal{L}} \frac{\Gamma(\zeta)\Gamma(1-\zeta)}{\Gamma(\mu - \alpha\zeta)} z^{-\zeta} d\zeta,
\end{equation}
where $\mathcal{L}$ is a suitable contour of integration in the complex plane (e.g., a Hankel contour separating the poles of $\Gamma(\zeta)$ and $\Gamma(1-\zeta)$). Substituting $z = \lambda |\xi|^{2s} (\gamma(t))^\alpha$ into \eqref{M-L as Mellin-Barnes} and these expresion in \eqref{eq:sol_fourier}, we obtain:
\begin{equation}
    \hat{\mathcal{G}}(\xi, t) = \frac{\rho(0^+)}{\rho(t)}(\gamma(t))^{\mu-1} \frac{1}{2\pi i} \int_{\mathcal{L}} \frac{\Gamma(\zeta)\Gamma(1-\zeta)}{\Gamma(\mu - \alpha\zeta)} \left[ \lambda (\gamma(t))^\alpha \right]^{-\zeta} |\xi|^{-2s\zeta} d\zeta.
\end{equation}
We now apply the weighted inverse Fourier transform, $\mathcal{F}^{-1}_{\phi, \omega}$, to recover the solution in the physical space. We assume that the interchange of the order of integration is permissible by virtue of Fubini's theorem. This assumption is supported by the asymptotic behaviour of the Gamma function, which provides sufficient exponential decay along the contour $\mathcal{L}$ to guarantee the absolute convergence of the integral.
\begin{equation}
    \mathcal{G}(x, t) = \frac{\rho(0^+)}{\rho(t)}\frac{(\gamma(t))^{\mu-1}}{2\pi i} \int_{\mathcal{L}} \frac{\Gamma(\zeta)\Gamma(1-\zeta) (\lambda (\gamma(t))^\alpha)^{-\zeta}}{\Gamma(\mu - \alpha\zeta)} \left[ \mathcal{F}^{-1}_{\phi, \omega} \left( |\xi|^{-2s\zeta} \right)(x) \right] d\zeta.
\end{equation}
The inner inverse transform corresponds to the Riesz potential in the weighted setting. Recalling the standard result for the inverse Fourier transform of a radial power $|\xi|^{-\beta}$ in $\mathbb{R}^n$:
\begin{equation}
    \mathcal{F}^{-1}(|\xi|^{-\beta})(y) = C_{n,\beta} |y|^{\beta - n}, \quad \text{where } C_{n,\beta} = \frac{2^{n-\beta}\pi^{n/2} \Gamma((n-\beta)/2)}{\Gamma(\beta/2)}.
\end{equation}
In our context, with exponent $\beta = 2s\zeta$ and transformed variable $y = \phi(x)$, the spatial term becomes:
\begin{equation}
    \mathcal{F}^{-1}_{\phi, \omega} \left( |\xi|^{-2s\zeta} \right)(x) = \frac{1}{\omega(x)} \frac{2^{n-2s\zeta}\pi^{n/2} \Gamma(n/2 - s\zeta)}{\Gamma(s\zeta)} |\phi(x)|^{2s\zeta - n}.
\end{equation}
Substituting this back into the contour integral and reorganising the terms, we arrive at:
\begin{equation}
    \mathcal{G}(x,t) = \frac{(\gamma(t))^{\mu-1}}{\omega(x) |\phi(x)|^n} (2\sqrt{\pi})^n \frac{1}{2\pi i} \int_{\mathcal{L}} \underbrace{ \frac{\Gamma(\zeta)\Gamma(1-\zeta)\Gamma(\frac{n}{2} - s\zeta)}{\Gamma(\mu - \alpha\zeta)\Gamma(s\zeta)} }_{\text{Kernel } \chi(\zeta)} \left[ \frac{2^{2s} \lambda (\gamma(t))^\alpha}{|\phi(x)|^{2s}} \right]^{-\zeta} \, d\zeta.
\label{eq:integral_rep}
\end{equation}
\textbf{In view of the definition of the inverse generalized Mellin transform given by (2.9)}, the expression \eqref{eq:integral_rep} constitutes precisely the inverse transform of the spectral function involving the Gamma factors. This establishes the validity of the representation in \textbf{Theorem 7.4} (Eq. 7.12).

Finally, to prove \textbf{Theorem 7.5}, we identify this integral with the definition of the Fox H-function:
\begin{equation}
    H_{p,q}^{m,n} \left[ z \middle| \begin{array}{c} (a_p, A_p) \\ (b_q, B_q) \end{array} \right] = \frac{1}{2\pi i} \int_{\mathcal{L}} \frac{\prod_{j=1}^m \Gamma(b_j + B_j \zeta) \prod_{j=1}^n \Gamma(1 - a_j - A_j \zeta)}{\prod_{j=m+1}^q \Gamma(1 - b_j - B_j \zeta) \prod_{j=n+1}^p \Gamma(a_j + A_j \zeta)} z^{-\zeta} d\zeta.
\end{equation}
Mapping the parameters from \eqref{eq:integral_rep}:
\begin{itemize}
    \item Numerator: $\Gamma(\zeta) \to (0,1)$, $\Gamma(1-\zeta) \to (0,1)$, $\Gamma(n/2 - s\zeta) \to (1-n/2, s)$.
    \item Denominator: $\Gamma(\mu - \alpha\zeta) \to (1-\mu, \alpha)$, $\Gamma(s\zeta) \to (0, s)$.
\end{itemize}
We conclude that the solution can be written compactly as:
\begin{equation} \label{eq:final_fox}
    \mathcal{G}(x,t) = \frac{(\gamma(t))^{\mu-1}}{|\phi(x)|^n} H^{1,2}_{3,2} \left[ \frac{|\phi(x)|}{(\lambda (\gamma(t))^\alpha)^{1/2s}} \, \Bigg| \, \begin{matrix} (0,1), (1 - \frac{n}{2}, s), (0,s) \\ (0,1), (1-\mu, \alpha) \end{matrix} \right].
\end{equation}
which completes the proof.
\end{proof}

\begin{remark}
\label{rem:fox_convergence}
\textbf{(Convergence and Parameter Validity).} The representation of the fundamental solution involves the Fox H-function defined via a Mellin-Barnes integral. For this expression to be rigorous, the parameters must satisfy the existence conditions established in the standard theory (see, e.g., Mathai et al. \cite{mathai2010}). Specifically, the integration contour $\mathcal{L}$ in the complex plane separates the poles of the Gamma functions $\Gamma(b_j + \beta_j s)$ from those of $\Gamma(1 - a_j - \alpha_j s)$. The absolute convergence of the integral is guaranteed by the condition $\sum_{j=1}^n \alpha_j - \sum_{j=1}^q \beta_j \le 0$ (and associated inequalities for the argument), which holds for the fractional diffusion parameters considered herein. This ensures that the fundamental solution is not only a formal expression but a well-defined function in the distributional sense.
\end{remark}

\begin{remark}[Universality via Conjugation and Time-Rescaling]
    The formulation presented in Theorem \ref{thm:fox_solution} offers a unified framework for anomalous diffusion. By appropriately selecting the time-scale function $\gamma(t)$ and a time-domain weight $\rho(t)$, our result encompasses the solutions for a wide range of fractional operators defined via \textit{conjugation relations}, as proposed by Fernandez and Fahad \cite{FernandezFahad2022_Fractal}.
    Specifically:
    \begin{itemize}
        \item The choice of $\gamma(t)$ allows recovering derivatives with respect to a function (e.g., $\psi$-Caputo or $\psi$-Hilfer).
        \item The introduction of a weight $\rho(t)$ corresponds to the \textbf{weighted Hilfer derivative with respect to a function}, defined essentially as $D_{\gamma, \rho}^{\alpha, \beta} f = \rho^{-1} D_{\gamma}^{\alpha, \beta} (\rho f)$.
    \end{itemize}
    In this context, the prefactor $(\gamma(t))^{\mu-1}$ in Eq. \eqref{eq:fox_final} can be interpreted as the natural kernel arising from the specific choice of the weight and scale functions in the generalized spectral domain.
\end{remark}
\begin{remark}[Comparison with Literature]
This result significantly generalizes recent works in two directions:
\begin{itemize}
    \item \textbf{Vs. Vieira et al. (2022) \cite{Vieira2022}:} They analyzed the diffusion equation involving the $\psi$-Hilfer derivative in standard Euclidean space. Our work extends this analysis to the fully weighted setting. We employ the weighted Hilfer-type derivative with respect to a function (based on the conjugation framework of Fernandez and Fahad \cite{FernandezFahad2022_Conj}) for the time evolution, and the weighted fractional Laplacian for the spatial component. Consequently, our fundamental solution captures both complex temporal memory effects and spatial heterogeneity via $\phi$ and $\omega$, generalizing their homogeneous space results.
    
    \item \textbf{Vs. Fernandez et al. (2024) \cite{Fernandez2024}:} While they successfully analyzed partial differential equations involving the fractional Laplacian with respect to a function $\phi$, their framework generally corresponds to the unweighted case ($\omega \equiv 1$). Our results generalize this analysis to the weighted setting, where the interplay between the geometry $\phi$ and the density $\omega$ necessitates the use of the spectral theory developed herein, leading to solutions governed by the generalized $\omega$-$\psi$-Mellin transform.
\end{itemize}
\end{remark}

\section{Conclusions}

In this work, we have successfully constructed a robust spectral theory for the Weighted Fourier Transform in $\mathbb{R}^n$. We established the fundamental harmonic analysis results within a weighted Hilbert space setting—including the Inversion Formula, Plancherel's Theorem, and Convolution—and proved a Heisenberg-type Uncertainty Principle, confirming the geometric consistency of the proposed transform.

Furthermore, we demonstrated the power of this framework by solving the generalized time-space fractional diffusion-wave equation involving the \textit{Weighted Hilfer Fractional Derivative}. Our results show that this spectral method allows for a unified analytical treatment that continuously interpolates between:
\begin{itemize}
    \item \textbf{Dynamic Regimes:} From sub-diffusion ($0 < \alpha < 1$) to wave propagation ($1 < \alpha \le 2$).
    \item \textbf{Initial Conditions:} From singular Riemann-Liouville problems ($\beta=0$) to regular Caputo problems ($\beta=1$).
\end{itemize}

Notably, the explicit solutions reveal an intrinsic connection with the Generalized Mellin Transform, confirming that the triad of transforms ($\mathcal{F}_{\phi,\omega}$, $\mathcal{L}_{\gamma,\rho}$, and $\mathcal{M}_{\phi,\omega}$) constitutes a complete and closed operational calculus for complex fractional systems. Significantly, the derived Uncertainty Principle revealed that the medium's geometry ($\phi$) dictates the fundamental limits of information localization, providing a spectral basis for the phenomenon of anomalous diffusion.

Finally, it is worth noting a promising avenue for future research regarding the geometric constraints. While this work focused on global diffeomorphisms with non-vanishing Jacobians to establish the spectral theory on $\mathbb{R}^n$, the explicit representation of the fundamental solution in terms of the Fox $H$-function suggests a natural extension to domains with singularities. Since $H$-functions possess rich asymptotic expansions near singularities—governed by the residues of the Gamma functions in the Mellin--Barnes integral—this framework could be adapted to handle weighted spaces where the Jacobian becomes singular or degenerate at the boundary. Investigating how the parameters of the $H$-function capture the asymptotic behavior of the solution near such geometric singularities remains an open and compelling problem. In addition to these theoretical extensions, future efforts will also focus on the numerical implementation of these operators and their application to bounded domains.

\end{document}